\documentclass{QT}

\address[kauffman@uic.edu]{Louis Hirsch Kauffman, Department of Mathematics, Statistics and Computer Science (m/c 249), 851 South Morgan Street, University of Illinois at Chicago,
Chicago, Illinois 60607-7045}
\address[vomanturov@yandex.ru]{Vassily Olegovich Manturov,
Chair of Differential Equations,
and Mathematical Physics, Faculty of Sciences,
Peoples' Friendship University of Russia,
117198, Moscow, Ordjonikidze St., 3}

\newtheorem{theorem}{Theorem}[section]
\newtheorem{corollary}[theorem]{Corollary}

\theoremstyle{definition}

\newtheorem{remark}[theorem]{Remark}

 \def\Z{{\mathbb Z}}
 \def\0{{\mathbbf 0}}
 \def\1{{\mathbbf 1}}

\newcommand{\skcrro}{\raisebox{-0.25\height}{\includegraphics[width=0.5cm]{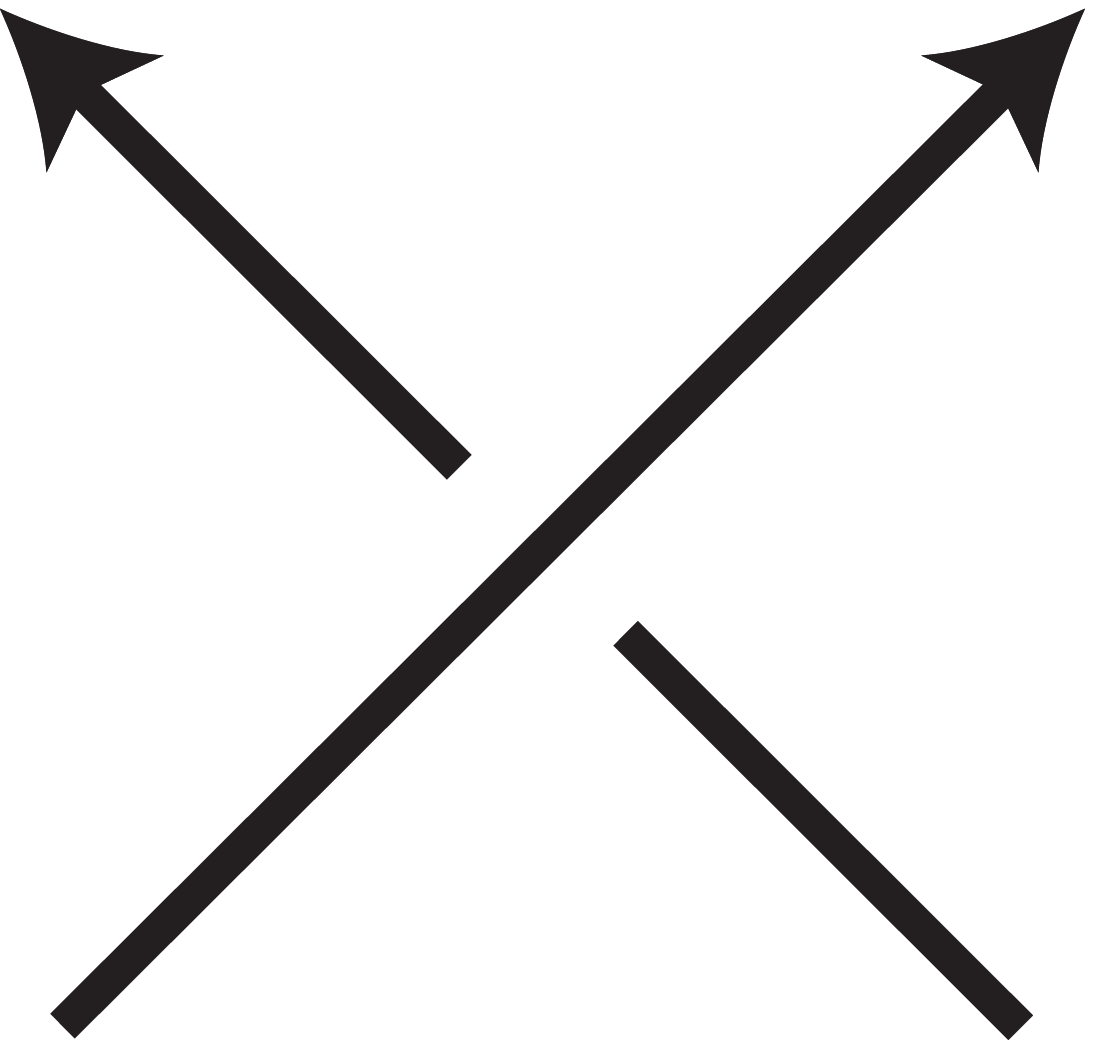}}}
\newcommand{\skcrlo}{\raisebox{-0.25\height}{\includegraphics[width=0.5cm]{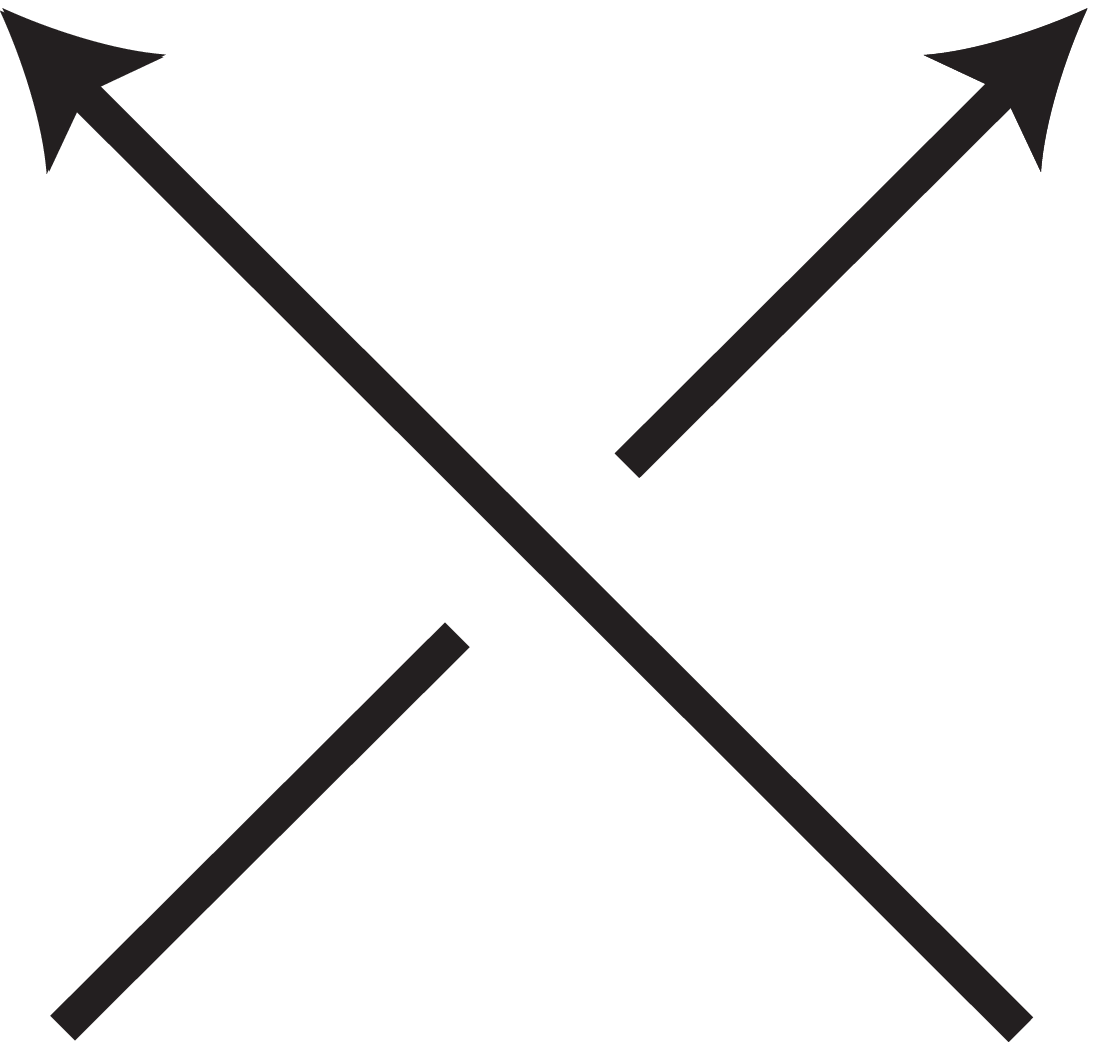}}}

\title[$sl(3)$ Invariant for Virtual Knots]{A Graphical Construction of the $sl(3)$ Invariant for Virtual Knots}

\author[Kauffman and Manturov]{Louis H.  Kauffman and Vassily O. Manturov \thanks{We would like to take this opportunity to thank the 
Mathematisches Forschungsinstitut Oberwolfach for their hospitality and wonderful research atmosphere. This paper was conceived and mostly written at the MFO in June of 2012 at a Research in Pairs of the present authors.
The second named author (V.O.M.) was partially supported by grants of the Russian Government
11.G34.31.0053,  RF President NSh 1410.2012.1, Ministry of Education and Science of the Russian Federation 14.740.11.0794.
}}

\begin{document}

\maketitle

\begin{abstract}
We construct a graph-valued analogue of the Homflypt
$sl(3)$ invariant for virtual knots. The restriction of this
invariant for classical knots coincides with the usual Homflypt
$sl(3)$ invariant, and for virtual knots and graphs it provides new
information that allows one to prove minimality theorems and to 
construct new invariants for free knots.  A novel feature of this approach is that some knots are 
of sufficient complexity that they evaluate themselves in the sense that the invariant is the knot 
itself seen as a combinatorial structure.
\end{abstract}

\begin{classification}
57M25.
\end{classification}

\begin{keywords}
Knot, link, virtual knot, graph, invariant, Kuperberg bracket, quantum invariant
\end{keywords}

\section{Introduction}
This paper studies a generalization to virtual knot theory of
the Kuperberg $sl(3)$ bracket invariant. Kuperberg  discovered a 
bracket state sum for the $sl(3)$ specialization of
the Homflypt polynomial that depends upon a reductive graphical procedure
similar to the Kauffman bracket but more complex. 
\bigbreak 

In this paper we show that the
Kuperberg bracket can be uniquely defined and generalized to virtual knot
theory via its reductive graphical equations. These equations reduce
to scalars only for the planar graphs from classical knots. For 
virtual knots, there are unique graphical reductions to linear
combinations of reduced graphs with Laurent polynomial coefficients.
Let us call these ``graph polynomials".
The ideal case, sometimes realized, is when the topological object is itself
the invariant, due to irreducibility. When this happens one can point to
combinatorial features of a topological object that must occur in all of
its representatives (first pointed out by Manturov in the context of parity).
This extended Kuperberg bracket specializes to
an invariant of free knots and allows us to prove that many free knots are
non-trivial without using the parity restrictions we had been tied to
before.  
\smallbreak

The aim of the present article is to extend the Kuperberg
combinatorial construction of the quantum $sl(3)$ invariant for the
case of virtual knots. In speaking of knots in this paper we refer to both knots and links.

In Figure~\ref{kuperberg}, we adopt the usual convention that whenever we give a
picture of a relation, we draw only the changing part of it; outside
the figure drawn, the diagrams are identical.
\smallbreak

\begin{figure}[htb]
     \begin{center}
     \begin{tabular}{c}
     \includegraphics[width=150pt]{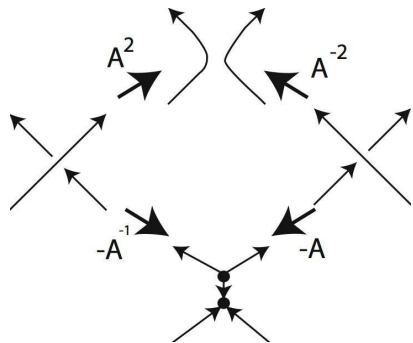}
     \end{tabular}
     \caption{\bf Kuperberg's relation for $sl(3)$}
     \label{kuperberg}
\end{center}
\end{figure}


For the case of the $sl(3)$ knot
invariant, one uses the relation shown in Figure~\ref{kuperberg}, see \cite{Kup}.
This means that the left (resp., right) picture of (\ref{kuperberg})
is resolved to a combination of the upper and lower pictures with
coefficients indicated on the arrows. The
advantage of Kuperberg's approach is that graphs of this sort which
can be drawn on the plane can be easily simplified, by using further
linear relations, to collections of Jordan curves, which in turn,
evaluate to elements from $\Z[A,A^{-1}]$. For planar graphs, these reductions continue all the way to scalars. In the case of non-planar graphs, there is no immediate way to
resolve such graphs to linear combinations of collections of
circles. We take it as an extra advantage of this approach that the non-planar resolutions leave irreducible graphs whose properties reflect the topology of virtual knots and links.
\smallbreak

The present paper is organized as follows. Section 2 is a review of concepts from virtual knot theory,
flat virtual knot theory and free knot theory.   Section 3 contains the construction of the main invariant in this paper, generalizing the Kuperberg bracket for $sl(3).$ Applications of this invariant to questions of minimality will be given elsewhere. Section 4 contains remarks about  the results in the paper and directions for future work.
\smallbreak

\section{Basics of Virtual Knot Theory, Flat Knots and Free Knots}
This section contains a summary of definitions and concepts in virtual knot theory that will be used in the rest of the paper.
\smallbreak

{\it Virtual knot theory} studies the  embeddings of curves in thickened surfaces of arbitrary
genus, up to the addition and removal of empty handles from the surface. See \cite{VKT,DVK}.  
Virtual knots have a special diagrammatic theory, described below,
that makes handling them
very similar to the handling of classical knot diagrams.  \smallbreak  

In the diagrammatic theory of virtual knots one adds 
a {\em virtual crossing} (see Figure~\ref{Figure 1}) that is neither an over-crossing
nor an under-crossing.  A virtual crossing is represented by two crossing segments with a small circle
placed around the crossing point. 
\smallbreak

Moves on virtual diagrams generalize the Reidemeister moves for classical knot and link
diagrams (Figure~\ref{Figure 1}).  The detour move is illustrated in 
Figure~\ref{Figure 2}.  The moves designated by (B) and (C) in Figure~\ref{Figure 1}, taken together, are equivalent to the detour move. Virtual knot and link diagrams that can be connected by a finite 
sequence of these moves are said to be {\it equivalent} or {\it virtually isotopic}.
A virtual knot is an equivalence class of virtual diagrams under these moves.
\smallbreak

\begin{figure}[htb]
     \begin{center}
     \begin{tabular}{c}
     \includegraphics[width=10cm]{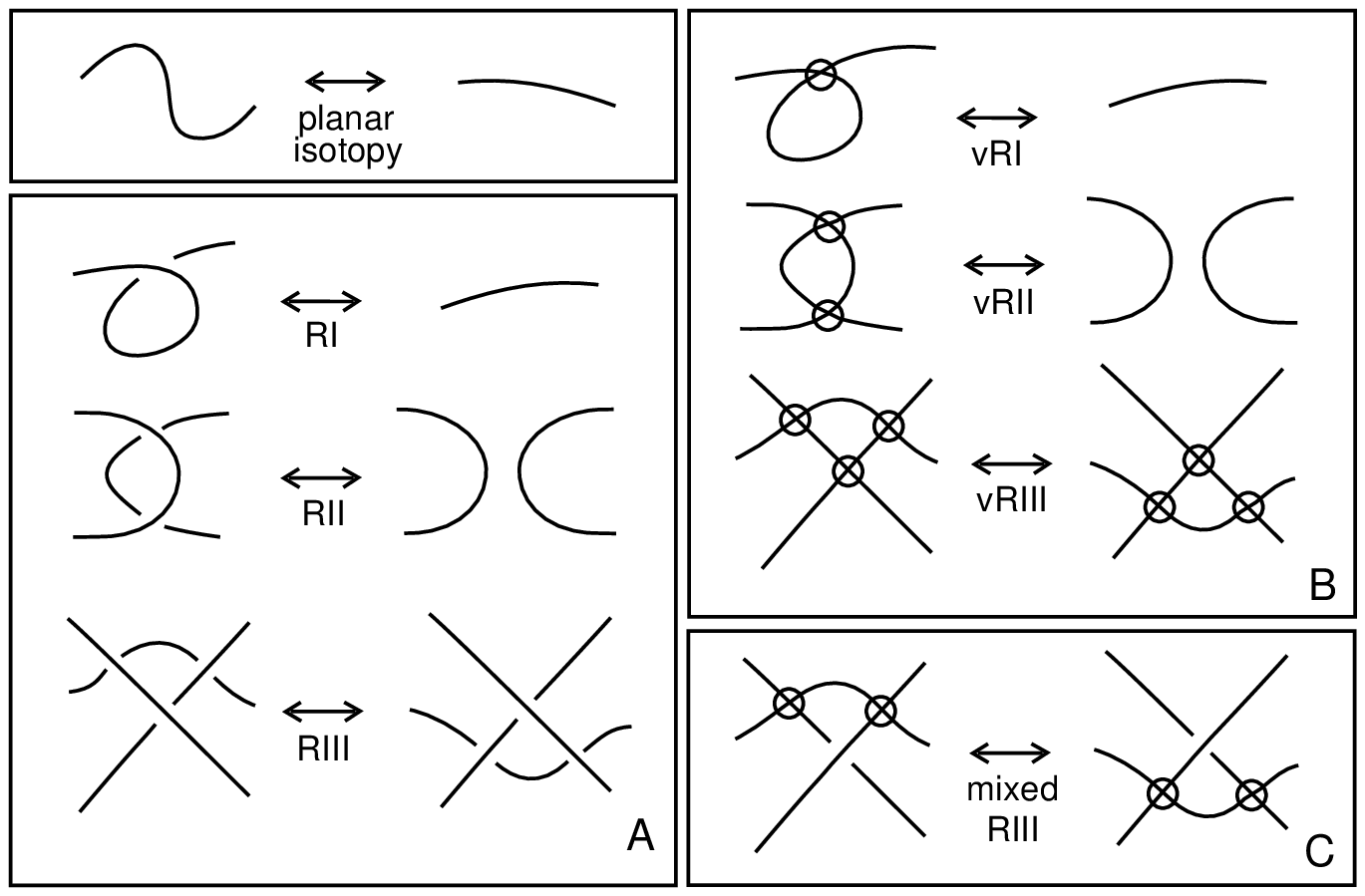}
     \end{tabular}
     \caption{\bf Moves}
     \label{Figure 1}
\end{center}
\end{figure}

\begin{figure}[htb]
     \begin{center}
     \begin{tabular}{c}
     \includegraphics[width=10cm]{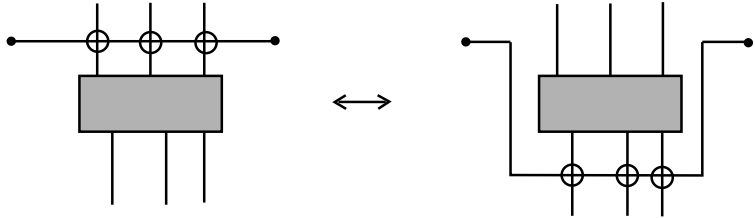}
     \end{tabular}
     \caption{\bf Detour Move}
     \label{Figure 2}
\end{center}
\end{figure}

Virtual diagrams can be regarded as representatives for oriented Gauss codes \cite{GPV}, \cite{VKT,SVKT} 
(Gauss diagrams). Such codes do not always have planar realizations.   {\it Virtual isotopy is the same as the equivalence relation generated on the collection
of oriented Gauss codes by abstract Reidemeister moves on these codes.}   The reader can see this approach in \cite{DKT,GPV,MB}. It is of interest to know the least number of virtual crossings that can occur in a diagram of a virtual knot or link. If this virtual crossing number is zero, then the link is classical. For some results about estimating virtual crossing number see \cite{DyeKauff,ExtBr,MV} and see the results of Corollaries $3$ and $4$ in Section $3$ of
the present paper.
\bigbreak

\begin{figure}
     \begin{center}
     \begin{tabular}{c}
     \includegraphics[width=5cm]{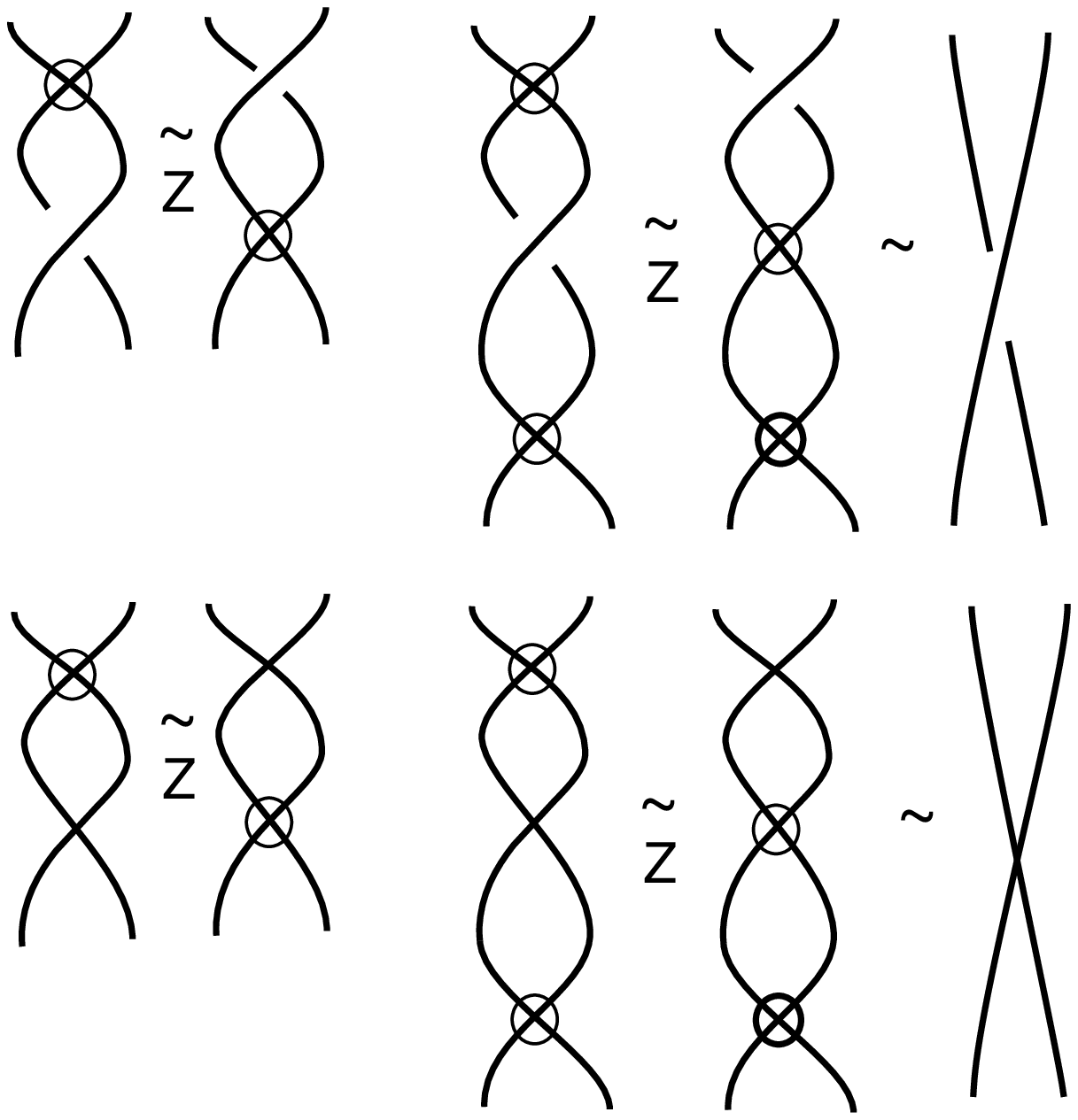}
     \end{tabular}
     \caption{\bf The Z-Move}
     \label{fig2}
\end{center}
\end{figure}

\noindent {\bf Flat and Free Knots and Links.}
Every classical knot diagram can be regarded as a $4$-regular plane graph with extra structure at the 
nodes. Let a {\em flat virtual diagram} be a diagram with {\it virtual crossings} as we have
described them and {\em flat crossings} consisting in undecorated nodes of the $4$-regular plane graph, retaining the cyclic order at a node. Two flat virtual diagrams are {\em equivalent} if
there is a  sequence of generalized flat Reidemeister moves (as illustrated in Figure~\ref{Figure 1}) taking one to the other. A generalized
flat Reidemeister move is any move as shown in Figure~\ref{Figure 1} where one ignores the over or under crossing structure. The moves for flat virtual knots are obtained by taking Figure~\ref{Figure 1} and replacing all the classical crossings by flat (but not virtual) crossings.
In studying flat virtuals the rules for changing virtual crossings among themselves and the rules for changing
flat crossings among themselves are identical. Detour moves as in part C of Figure~\ref{Figure 1} are available for virtual crossings
with respect to flat crossings and {\it not} the other way around. 
\smallbreak

To each virtual diagram $K$ there is an associated 
flat diagram $F(K)$, obtained by forgetting the extra structure at the classical crossings in $K.$ 
We say that a virtual diagram {\em overlies} a flat diagram if the virtual diagram is obtained from the flat diagram by choosing a crossing type for each flat crossing in the virtual diagram. 
If $K$ and $K'$
are isotopic as virtual diagrams, then $F(K)$ and $F(K')$ are isotopic as flat virtual diagrams. Thus, if we can
show that $F(K)$ is not reducible to a disjoint union of circles, then it will follow that $K$ is a non-trivial  and non-classical virtual link.   
\smallbreak 

\noindent {\bf Definition.}
A virtual graph is a $4-regular$ graph that is immersed in the plane giving a choice of cyclic orders at its nodes.  The edges at the nodes are connected according to the abstract definition of the graph and are embedded into the plane so that they intersect transversely. These intersections are taken as virtual crossings and are subject to the detour move just as in virtual link diagrams.  
We allow circles along with the graphs of any kind in our work with graph theory.
\smallbreak

\noindent  {\bf Framed Nodes and Framed Graphs.}
We use the concept of a {\it framed $4$-valent node} where we only specify the pairings of {\it opposite edges} at the node. In the cyclic order, two  edges are said to be opposite if they are paired by skipping one edge as one goes around. If the cyclic order of a node is $[a,b,c,d]$ where these letters label the edges incident to the node, then we say that edges $a$ and  $c$ are {\it opposite}, and that edges
$b$ and $d$ are {\it opposite}.  We can change the cyclic order and keep the opposite relation.
For example, in $[c,b,a,d]$ it is still the case that the opposite pairs are $a,c$ and $b,d.$ A {\it framed $4$-valent graph} is a $4$-valent graph where every node is framed. When we represent a framed $4$-valent graph as an immersion in the plane, we use virtual crossings for the edge-crossings that are artifacts of the immersion and we regrad the graph as a virtual graph. For an abstract framed $4$-valent graph, there are no classical crossings - the only Reidemeister moves that occur are among the virtual crossings.
\smallbreak

A  {\it component} of a framed graph is obtained by taking a walk on the graph so that the walk contains pairs of opposite edges from every node that is met during the walk. That is, in walking, if you enter a node along a given edge, then you exit the node along its opposite edge. Such a walk produces a cycle in the graph and such cycles are called the {\it components} of the framed graph. Since a link diagram or a flat link diagram is a framed graph, we see that the components of this framed graph are identical with the components of the link as identified  by the topologist. A framed graph with one component is said to be {\it unicursal}.  
\smallbreak

When we take virtual knot diagrams only up to framing of their classical nodes, we are allowing the $Z-move$ as illustrated in Figure~\ref{fig2}.
In the $Z-move$ one can intechange a crossing with an adjacent virtual crossing, as shown in the figure.  We call virtual knots and links modulo the $Z$-move, {\it Z-knots}. We call flat virtual knots modulo the 
Z-move, {\it free knots}. Thus free knots are the same as framed $4$-valent graphs taken up to the flat
Reidemeister moves. 
{\it The theory of free knots is identical with the theory that one gets when one takes Flat Virtuals modulo the flat $Z$-move} as shown in Figure~\ref{fig2}. We say that a {\it virtualization move} has been performed on a crossing if it is flanked by two virtual crossings. We illustrate this operation in Figure~\ref{fig2} and show that virtualization does not change the equivalence class of a flat diagram under the $Z$-move. This means that any invariant of free knots must be invariant under virtualization.
\smallbreak

\section{Construction of the Main Invariant}

Let ${\cal S}$ be the collection of all trivalent bipartite graphs
with edges oriented from vertices of the first part to vertices of
the second part.
Let ${\cal T}=\{t_{1},t_{2},\cdots\}$ be the (infinite) subset of
connected graphs from ${\cal S}$ having neither bigons nor quadrilaterals.
Let ${\cal M}$ be the module $\Z[A,A^{-1}][t_{1},t_{2},\cdots]$ of
formal commutative products of graphs from
${\cal T}$ with coefficients that are Laurent polynomials.
Disjoint unions of graphs  are treated as products in ${\cal M}$.
Our main invariant will be valued in ${\cal M}$.
\smallbreak

\noindent {\bf Statement 1.}
{\it Figure~\ref{kupbra} shows the reduction moves for the Kuperberg bracket. The last three lines of the figure will be called the {\it relations} in that Figure.
There exists a unique map $f:{\cal S}\to {\cal M}$ which satisfies
the relations in Figure~\ref{kupbra}. Note that we have shown part of these relations in Figure~\ref{kuperberg}. The resulting evaluation yields a topological invariant of virtual links when the first two lines of Figure~\ref{kupbra}  are used to expand the link into a sum of elements of ${\cal S}.$}
\bigbreak

\begin{figure}[htb]
     \begin{center}
     \begin{tabular}{c}
     \includegraphics[width=6cm]{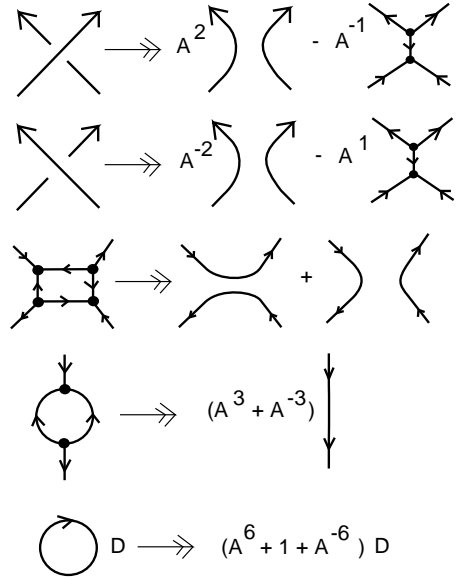}
     \end{tabular}
     \caption{\bf Kuperberg Bracket}
     \label{kupbra}
\end{center}
\end{figure}

\begin{proof}
The relations we are going to use to prove the statement are as shown in Figure~\ref{kupbra}.
Note that for the case of planar tangles this map to diagrams modulo relations was constructed
explicitly by Kuperberg \cite{Kup}, and the image was in
$\Z[A,A^{-1}]$. We are going to follow \cite{Kup}, however, in the
non-planar case, the graphs can not be reduced just to collections
of closed curves (in the case of the plane, Jordan curves) and so
later evaluate to polynomials. In fact, irreducible graphs will appear in the non-planar case.
First, we treat every 1-complex with all components being graphs
from ${\cal S}$ and circles: We treat it as the formal product of
these graphs, where each circle evaluates to the factor
$(A^{6} + A^{-6} +1)$.
We note that if a graph $\Gamma$ from ${\cal S}$ has a bigon
or a quadrilateral, then we can use the relations shown in Figure~\ref{kupbra}
(resolution of quadrilaterals, resolution of bigons, loop evaluation)
to reduce it to a smaller graph (or two graphs, then we consider it
as a product).
So, we can proceed with resolving bigons and quadrilaterals until we
are left with a collection of graphs $t_{j}$ and circles; this gives
us an element from ${\cal M}$; once we prove the uniqueness of the resolution, we set the
stage for proving the existence of the invariant. We must carefully check well-definedness and topological invariance.
\smallbreak

In what follows, we shall often omit the letter
$f$ by identifying graphs with their images or intermediate graphs
which appear after some concrete resolutions.
\smallbreak

Our goal is to show that this map $f:{\cal S}\to {\cal M}$ is
well-defined. We shall prove it by induction on the number of graph edges.
{\it The induction base is obvious} and we leave its articulation to the reader.
To perform the induction step, notice that all of Kuperberg's relations
are {\it reductive}: from a graph we get to a collection of simpler
graphs.
\smallbreak

Assume for all graphs with at most $2n$ vertices that the statement
holds. Now, let us take a graph $\Gamma$ from ${\cal S}$ with $2n+2$
vertices. Without loss of generality, we assume this graph is
connected. If it has neither bigon nor quadrilateral, we just take
the graph itself to be its image.
Otherwise we use the relations {\it resolution of bigons} or
{\it resolution of quadrilaterals} as in Figure~\ref{kupbra}  to reduce it to a linear combination of
simpler graphs; we proceed until we have a sum (with Laurent polynomial coefficients) of
(products of) graphs without bigons and quadrilaterals.
\smallbreak

According to the induction hypothesis, for all simpler graphs, there
is a unique map to ${\cal M}$. However, we can apply the relations
in different ways by starting from a given quadrilateral or a bigon.
We will show that the final result does not depend on the bigon or
quadrilateral we start with.
To this end, it suffices to prove that if $\Gamma$ can be resolved
to $\alpha \Gamma_{1}+\beta \Gamma_{2}$ from one bigon
(quadrilateral) and also to $\alpha' \Gamma'_{1}+\beta' \Gamma'_{2}$ from the other one, then
both linear combinations can be resolved further, and will lead to
the same element of ${\cal M}$. This will show that final reductions are unique.
\smallbreak

Whenever two nodes of a quadrilateral coincide, then two edges coincide and it is 
no longer subject to the quadrilateral reduction relation. Thus we assume that quadrilaterals under discussion have distinct nodes.
Note that if two polygons (bigons or
quadrilaterals) share no common vertex then the corresponding two
resolutions can be performed {\em independently} and, hence, the
result of applying them in any order is the same. So, in this case,
$\alpha \Gamma_{1}+\beta \Gamma_{2}$ and $\alpha' \Gamma'_{1}+\beta'
\Gamma'_{2}$ can be resolved to the same linear combination in one
step. By the hypothesis,
$f(\Gamma_{1}),f(\Gamma_{2}),f(\Gamma'_{1}),f(\Gamma'_{2})$ are all
well defined, so, we can simplify the common resoltion for $\alpha
\Gamma_{1}+\beta \Gamma_{2}$ and $\alpha' \Gamma'_{1}+\beta'
\Gamma'_{2}$ to obtain the correct value for $f$ of any of these two
linear combinations, which means that they coincide.
\smallbreak

If two polygons (bigons or quadrilaterals) share a vertex, then
they share an edge because the graph is trivalent.
If  a connected trivalent graph has two different bigons sharing an
edge then the total number of edges of this graph is three, and the
evaluation of this graph in ${\cal T}$ follows from an easy
calculation.
Therefore, let us assume we have a graph $\Gamma$ with an edge shared by a
bigon and a quadrilateral. We can resolve the quadrilateral first,
or we can resolve the bigon first.
The calculation in Fig. \ref{undobigonsquare}
\begin{figure}
\centering\includegraphics[width=250pt]{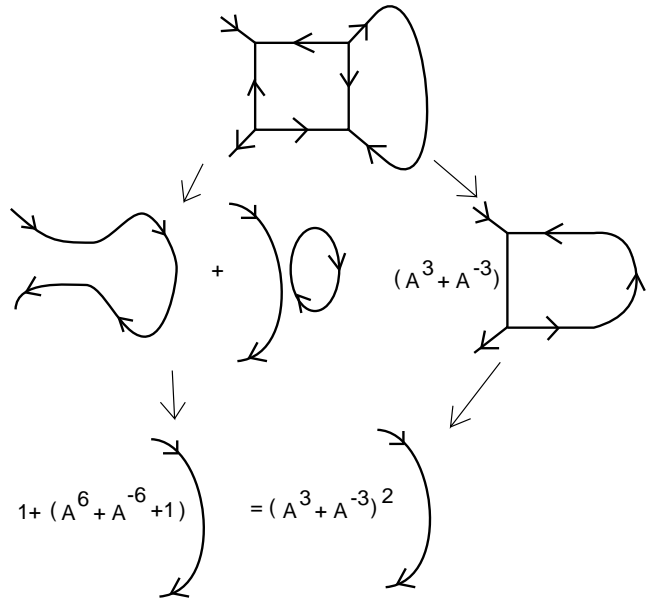}
\caption{\bf Two Ways of Reducing a Quadrilateral and a Bigon}
\label{Resolving a bigon and a quadrilateral}
\label{undobigonsquare}
\end{figure}
shows that after a two-step resolution we get to the same linear
combination.
\smallbreak

A similar situation happens when we deal with two quadrilaterals
sharing an edge, see Fig. \ref{undotwosquares}. Here we have shown
just one particular resolution, but the picture is symmetric, so the
result of the resolution when we start with the right quadrilateral,
will lead us to the same result. See also Fig. \ref{triang} and Fig. \ref{annular}.
These figures illustrate two other ways in which the edge can be shared. Note that
Fig. \ref{triang} illustrates a possibly non-planar case, and that we use the abstract graph structure
(no particular order at the trivalent vertex) in the course of the evaluation. These cases 
cover all the ways that shared edges can occur, as the reader can easily verify.

\begin{figure}
\centering\includegraphics[width=210pt]{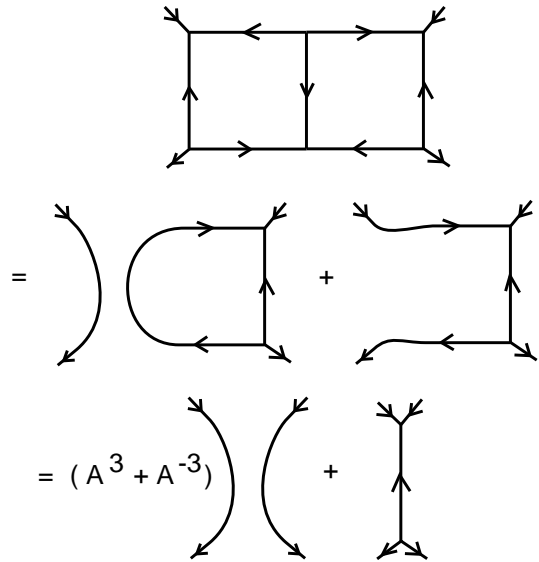}
\caption{\bf Resolving Two Adjacent Squares} \label{Resolving two
squares} \label{undotwosquares}
\end{figure}

\begin{figure}
\centering\includegraphics[width=200pt]{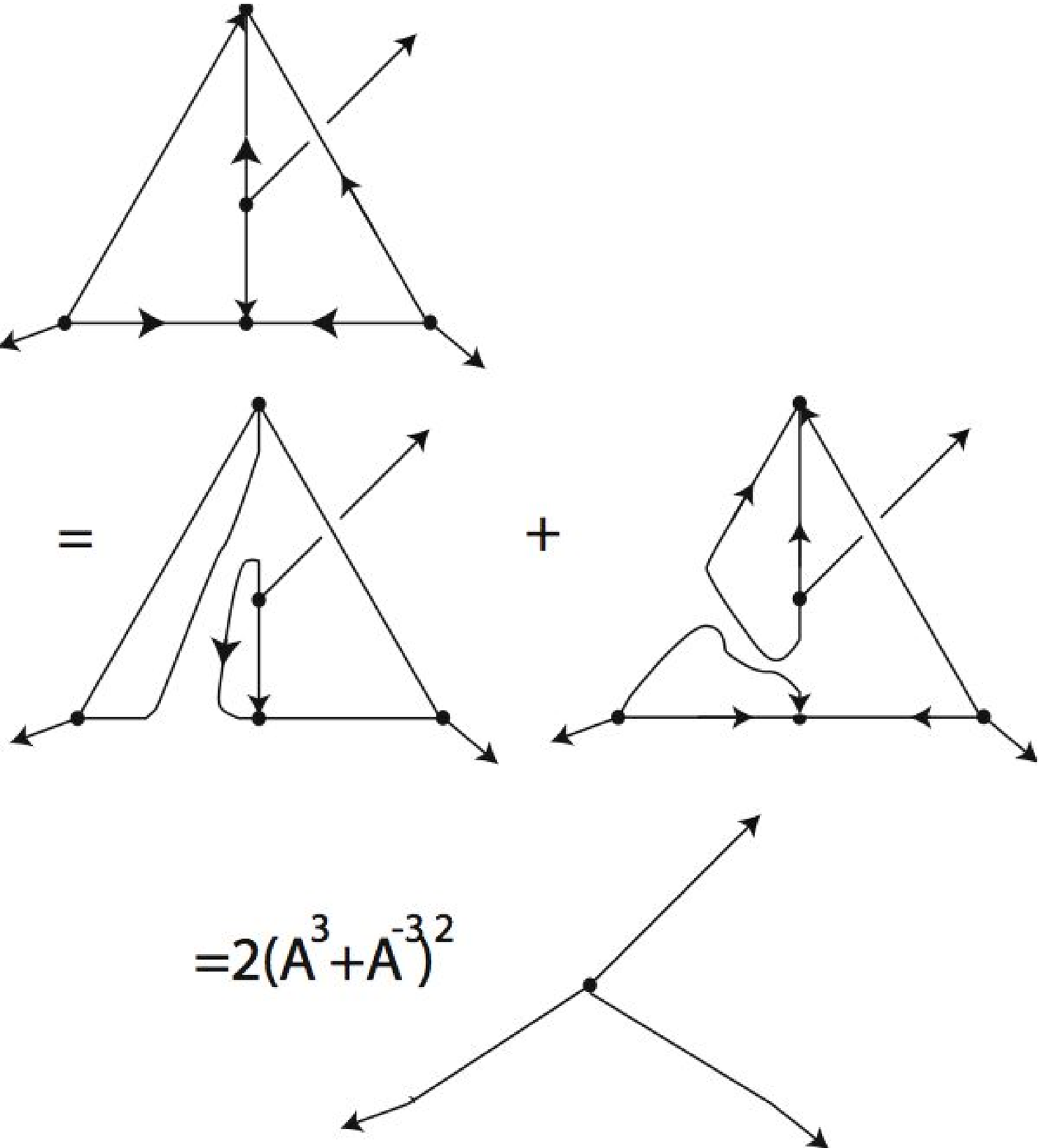}
\caption{\bf Resolving Two Different Adjacent Squares} \label{triang} \label{triang}
\end{figure}

\begin{figure}
\centering\includegraphics[width=200pt]{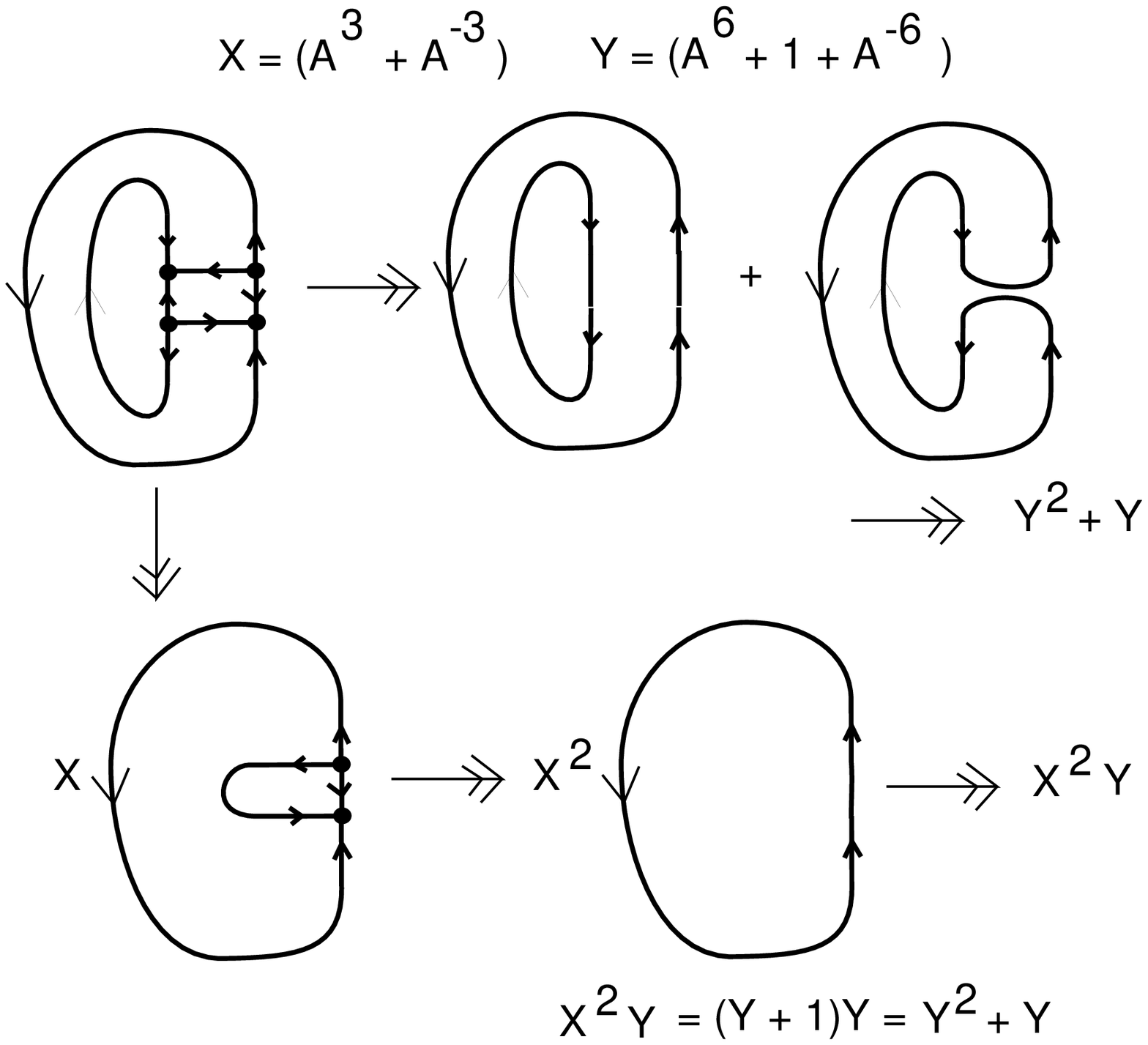}
\caption{\bf Resolving Two Annular Squares} \label{annular} \label{annular}
\end{figure}

Thus, we have performed the induction step and proved the
well-definiteness of the mapping. 
Note that the ideas of the proof are the same as in the classical case; however, we never assumed any planarity of the
graph; we just drew graphs planar whenever possible. Note that the situation in Figure~\ref{triang} is principally non-planar.
The invariance under virtualization  follows from this definition because the graphical pieces into which we expand a crossing, as in Figure~\ref{muhoya}, are, as graphs, symmetric under the interchange produced by the virtualization.
\end{proof}

\noindent {\bf Remark.} We can, in the case of flat knots or standard virtual knots represented on surfaces, enhance the invariant by keeping track of the embedding of the graph in the surface and only expanding on bigons and quadrilaterals that bound in the surface. We will not pursue this version of the
invariant here. In undertaking this program we will produce evaluations that are not invariant under the 
$Z$-move for flats or for virtual knots. 
\smallbreak

Now we give a formal description of our main invariant. This evaluation is invariant under the 
$Z-move.$ It is defined for virtual knots and links and it specializes to an invariant of free knots.
Let $K$ be an oriented virtual link diagram. With every classical
crossing of $K$, we associate two local states: the {\it oriented} one and
the {\it unoriented} one: the oriented one shown in the upper picture of
Fig. \ref{kuperberg}, and the unoriented one shown in the lower
picture of Fig. \ref{kuperberg}. A {\it state} of the diagram is a choice of local state for each crossing in the diagram.
\smallbreak

We define the bracket $[[\cdot]]$ (generalized Kuperberg bracket) as follows. Let $K$ be an
oriented virtual link diagram.
For a state $s$ of a virtual knot diagram $K$, we define the weight
of the state as the coefficient of the corresponding graph according
to the Kuberberg relations (\ref{kuperberg}). More precisely, the weight of a
state is the product of weights of all crossings, where a weight of
a positive crossing is $A^{ 2wr}$ for the oriented resolution and
$-A^{- wr}$ for the unoriented resolution, $wr$ stands for the writhe
number (the oriented sign) of the crossing.

Set
\begin{equation}
[[K]] = \sum_{s} w(K_{s})\cdot f(K_{s}) \in {\cal M},
\end{equation}

where $w(s)$ is the weight of the state.

\begin{theorem}
For a given diagram $K,$ the normalized bracket $(A^{-8wr(K)})[[K]]$ is invariant under all Reidemeister moves
and the virtualization move. Here $wr(K)$ denotes the writhe obtained by summing the signs of all the classical crossings in the corresponding diagram.
\end{theorem}

\begin{proof}
The invariance proof under Reidemeister moves repeats that of Kuperberg.
Note that the writhe behaviour is a consequence of the relations in Figure~\ref{kupbra}.
The only thing we require is that the Kuperberg relations
(summarized in Figure~\ref{kupbra}) can be applied to yield unique reduced graph polynomials.
 The discussion preceding the proof, proving Statement $1$,  handles this issue.
\end{proof}

From the definition of $[[K]]$ we have the following:
\begin{corollary}
If $[[K]]$ does not belong to $\Z[[A,A^{-1}]]\subset {\cal M}$ then
the knot $[[K]]$ is not classical.
\end{corollary}

Recalling that a free link is an equivalence of virtual knots modulo
virtualizations and crossing switches and taking into account that
the skein relations in Figure~\ref{kuperberg} for $[[\cdot]]$ for $\skcrro$
and $\skcrlo$ are the same when specifying $A=1$, we get the
following

\begin{corollary}
$[[K]]_{A=1}$ and $[[K]]_{A= -1}$ are invariants of free links.
\end{corollary}

By the unoriented state $K_{us}$ of virtual knot diagram (resp.,
free knot diagram) $K$ we mean the state of $K$ where all crossings
are resolved in a way where an edge is added. {\bf Notation:}
$K_{us}$. Note that $K_{us}$ is treated as a graph.

\begin{corollary} Assume for a
virtual knot (or free knot) $K$ with $n$ classical crossings the
graph $K_{us}$ has neither bigons nor quadrilaterals. Then every
knot $K'$ equivalent to $K$ has a state $s$ such that $K'_{s}$
contains $K_{us}$ as a subgraph. This state can be treated as an element of ${\cal M}$.
In particular, $K$ is minimal, and all minimal diagrams of this free
knot have the same number of crossings.
 \label{crl1}
\end{corollary}

Note that the coincidence of $K_{us}$ and $K'_{us}$ does not
guarantee the coincidence of $K$ and $K'$. For example, if $K$ and
$K'$ differ by a third unoriented Reidemeister move, then, of
course, $[[K]]=[[K']]$. The corresponding resolutions $K_{us}$ and
$K'_{us}$ will coincide (they will have a hexagon inside).
 
\begin{corollary}
Let $K$ be a four-valent framed graph with $n$ crossings and with
girth number at least five.
Then the hypothesis of Corollary \ref{crl1} holds. \label{crl2}
\end{corollary}

So, this proves the minimality of a large class of framed
four-valent graphs regarded as free knots: all graphs having girth
$\ge 5$ and many other knots. For example, consider the free knot $K_{n}$
whose Gauss diagram is the $n$-gon, $n>6$: it consists of $n$ chords
where $i$-th chord is linked with exactly two chords, those having
numbers $i-1$ and $i+1$ (the numbers are taken modulo $n$). Then $K_{n}$
satisfies the condition of \ref{crl1} and, hence, is minimal in a
strong sense.

Note that the triviality of such $n$-gons as free knots was proven
only for $n\le 6$.

\begin{remark}
The above argument works for links and tangles as well as knots.
\end{remark}

From the construction of $[[\cdot]]$ we get the following corollary.
\begin{corollary}
Let $K$ be a virtual (resp., flat) knot, and let
$\Gamma_{1}\cdots\Gamma_{k}$ be a product of irreducible
graphs which appear as a summand in $[[K]]$ (resp.,
$[[K]]|_{A=1}$) with a non-zero coefficient.
 Then the minimal virtual crossing number of $K$ is greater than or
equal to the sum of crossing numbers of graphs: $cr(\Gamma_{1})+\cdots + cr(\Gamma_{k})$ and the underlying genus of $K$ is
less than or equal to the
sum of genera $g(\Gamma_{1})+\cdots + g(\Gamma_{k})$
(in virtual or free knot category).
\end{corollary}

The above corollary easily allows one to reprove the theorem first proved in \cite{MV},
that the number of virtual crossings of a virtual knot grows quadratically with respect
to the number of classical knots for some families of graphs. In \cite{MV}, it was
done by using the parity bracket. Now, we can do the same by using $$Free[[K]] = [[K]]|_{A=1}.$$
With this invariant one can easily construct infinite series of trivalent bipartite graphs which serve as $K_{us}$ for some sequence of knots $K_{n}$ and such
that the minimal crossing number for
these graphs grows quadratically with respect to the number of crossings. Recalling that
the number of vertices comes from the number of classical crossings of $K_{n}$, we
get the desired result. 
\smallbreak

\section  {\bf Remarks} 
This article arose through our discussions of new possibilities in virtual knot theory and
in relation to advances of Manturov using parity in
virtual knot theory, particularly in the area of free knots. 
Manturov was the first person to show that many free
knots are non-trivial.   A free
knot is a Gauss diagram with only chords and without signs on the chords
or  orientations on them. Such Gauss diagrams are taken up to Reidemeister
moves and they underlie the structures of virtual knot theory.  
\bigbreak

The constructed invariant has the following properties.

\begin{enumerate}

\item It coincides with the usual $sl(3)$ quantum invariant in the case
of the classical knots.

\item It does not change under virtualization (i.e. under the $Z$-move as defined in this paper) ; its specification at $A=1$ gives rise to an
invariant of free knots.

\item For virtual knots, that are complcated enough, the new  invariant is valued
in a certain module whose generators are graphs.

\item The invariant produces many new examples of minimality in a strong sense
for free knots in the flavour of the parity bracket, \cite{Sbornik1,Parity}. But it goes beyond simple parity and
can discriminate certain free knots that have only even crossings.
\end{enumerate}

In the paper \cite{MOY}, there is a model for the $sl(n)$-version of
the Homflypt polynomial for classical knots. This model is  based on patterns of smoothings as shown
in Figure~\ref{muhoya}.

\begin{figure}[htb]
     \begin{center}
     \begin{tabular}{c}
      \includegraphics[width=100pt]{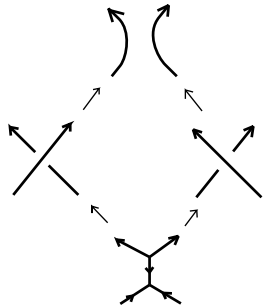}
     \end{tabular}
     \caption{\bf Murakami-Ohtsuki-Yamada relation for $sl(n)$}
     \label{muhoya}
\end{center}
\end{figure}


These patterns suggested to us the techniques we use in this paper with the Kuperberg bracket, and we expect to generalize them further.
In this method,  the value of the polynomial for a knot
is equal to the linear combination of the values for two graphs
obtained from the knot by resolving the two crossings as shown in Figure~\ref{muhoya}. 

In the present paper we enhance the Homflypt $sl(3)$ invariant by
using the following observation: {\it if a trivalent graph is complicated
enough so that it admits no further simplification, it can be
evaluated as itself. } Then the obtained ``polynomial'' $sl(3)$ invariant
 will be valued not just in Laurent polynomials in one
variable $A$, but  in a larger ring where trivalent graphs act as
variables. A key point here is that it can be the case that a topological object such as a free knot is its 
{\it own} invariant!  See \cite{AC} for a use of this idea. This is what happens when we meet irreducibility in the graphical expansion of our invariant. Then it is possible for the expansion to simply stop at the object itself.  This rigidity occurs when the non-planar graphs in the expansion of the generalized Kuperberg bracket are irreducible. Then these graphs are valued as themselves, rather than as polynomials in other graphs.
\bigbreak

\end{document}